\documentclass[reqno,a4paper]{amsart} 
\usepackage[utf8]{inputenc}
\usepackage[T1]{fontenc}
\usepackage[british]{babel}
\usepackage{amsfonts, amssymb, amsmath}
\usepackage[all]{xy}
\usepackage{hyperref}
\usepackage{mathbbol}
\usepackage{calrsfs}
\usepackage{mathabx}
\usepackage{xcolor}
 
\theoremstyle{plain}
\newtheorem{theorem}[subsection]{Theorem}
\newtheorem{proposition}[subsection]{Proposition}
\newtheorem{corollary}[subsection]{Corollary}	
\newtheorem{lemma}[subsection]{Lemma}

\theoremstyle{definition}
\newtheorem{definition}[subsection]{Definition}

\theoremstyle{remark}
\newtheorem{remark}[subsection]{Remark}
\newtheorem{example}[subsection]{Example}
\newtheorem{examples}[subsection]{Examples}
\newtheorem{problem}[subsection]{Open problem}

\numberwithin{equation}{section}

\newenvironment{tfae}
{
\begin{enumerate}}
{\end{enumerate}}

\newcommand{\defn}{\textbf}
\newcommand{\noproof}{\hfil\qed}
\newcommand{\bigjoin}{\bigvee}
\newcommand{\join}{\vee}
\newcommand{\meet}{\wedge}
\newcommand{\cosmash}{\diamond}
\newcommand{\mus}[2]{\left\lgroup #1 \; #2 \right\rgroup}
\newcommand{\normal}{\ensuremath{\lhd}}

\DeclareMathOperator*{\colim}{Colim}
\DeclareMathOperator{\cod}{Cod}
\DeclareMathOperator{\Ker}{Ker}
\DeclareMathOperator{\Aut}{Aut}

\newcommand{\B}{\ensuremath{\mathcal{B}}}
\newcommand{\C}{\ensuremath{\mathcal{C}}}
\newcommand{\D}{\ensuremath{\mathcal{D}}}

\newcommand{\Act}{\ensuremath{\mathsf{Act}}}
\newcommand{\Alg}{\ensuremath{\mathsf{Alg}}}
\newcommand{\Arr}{\ensuremath{\mathsf{Arr}}}
\newcommand{\Cat}{\ensuremath{\mathsf{Cat}}}

\newcommand{\Eq}{\ensuremath{\mathsf{Eq}}}
\newcommand{\Ext}{\ensuremath{\mathsf{Ext}}}
\newcommand{\CExt}{\ensuremath{\mathsf{CExt}}}
\newcommand{\Gp}{\ensuremath{\mathsf{Gp}}}

\newcommand{\Mon}{\ensuremath{\mathsf{Mon}}}

\newcommand{\PXMod}{\ensuremath{\mathsf{PXMod}}}
\newcommand{\RG}{\ensuremath{\mathsf{RG}}}

\newcommand{\SPt}{\ensuremath{\mathsf{SPt}}}
\newcommand{\XMod}{\ensuremath{\mathsf{XMod}}}

\newcommand{\LACC}{{\rm (LACC)}}

\newcommand{\K}{\mathbb{K}}
\newcommand{\N}{\mathbb{N}}
\newcommand{\V}{\mathcal{V}}
\newcommand{\Z}{\mathbb{Z}}

\newcommand{\Pf}{\mathcal{P}_{\text{Fin}}}
\renewcommand{\Im}{\mathrm{Im}}
\newcommand{\im}{\mathrm{im}}
\newcommand{\Sub}[2]{\textnormal{$\mathsf{Sub}_{#2}(#1)$}}
\newcommand{\act}[1]{\textnormal{$#1$-$\mathsf{Act}$}}
\newcommand{\PT}[1]{\textnormal{$\mathsf{Pt}(#1)$}}
\newcommand{\Pt}[2]{\textnormal{$\mathsf{Pt}_{#2}(#1)$}}
\newcommand{\SplExt}[2]{\textnormal{$\mathsf{SplExt}_{#2}(#1)$}}
\newcommand{\comp}{\circ}

\newdir{>}{{}*:(1,-.2)@^{>}*:(1,+.2)@_{>}}
\newdir{<}{{}*:(1,+.2)@^{<}*:(1,-.2)@_{<}}
\newdir{>>}{{}*!/3.5pt/:(1,-.2)@^{>}*!/3.5pt/:(1,+.2)@_{>}*!/7pt/:(1,-.2)@^{>}*!/7pt/:(1,+.2)@_{>}}
\newdir{ >}{{}*!/-8pt/@{>}}

\begin{document}

\title{Algebraic logoi}

\author{D.\ Bourn}
\author{A.\ S.\ Cigoli}
\author{J.\ R.\ A.\ Gray}
\author{T.~Van~der Linden}

\email{dominique.bourn@univ-littoral.fr}
\email{alan.cigoli@unito.it}
\email{jamesgray@sun.ac.za}
\email{tim.vanderlinden@uclouvain.be}

\address[Dominique Bourn]{Univ.\ Littoral C\^ote d'Opale, UR 2597, LMPA, Laboratoire de Math\'ematiques Pures et Appliqu\'ees Joseph Liouville, F--62100 Calais, France}
\address[Alan S.\ Cigoli]{Universit\`a degli Studi di Torino, Dipartimento di Matematica ``G.\ Peano'', via Carlo Alberto 10, I--10123 Torino, Italy}
\address[James R.\ A.\ Gray]{Mathematics Division, Department of Mathematical Sciences, Stellenbosch University, Private Bag X1, 7602 Matieland, South Africa}
\address[Tim Van~der Linden]{Institut de Recherche en Math\'ematique et Physique, Universit\'e catholique de Louvain, che\-min du cyclotron~2 bte~L7.01.02, B--1348 Louvain-la-Neuve, Belgium}

\thanks{The fourth author is a Senior Research Associate of the Fonds de la Recherche Scientifique--FNRS}

\begin{abstract}
	We introduce \emph{normal cores}, as well as the more general \emph{action cores}, in the context of a semi-abelian category, and further generalise those to \emph{split extension cores} in the context of a homological category. We prove that, if the category is moreover well-powered with (small) joins, then the existence of split extension cores is equivalent to the condition that the change-of-base functors in the fibration of points are \emph{geometric}. We call a finitely complete category that satisfies this condition an \emph{algebraic logos}. We give examples of such categories, compare them with algebraically coherent ones, and study equivalent conditions as well as stability under common categorical operations.
\end{abstract}


\maketitle

\section{Introduction}
The aim of this article is to introduce \emph{normal cores}, \emph{action cores} and \emph{split extension cores} in the context of homological and semi-abelian categories~\cite{Janelidze-Marki-Tholen,Borceux-Bourn}. For this to work well, we require the setting of what we call an \emph{algebraic logos}: a finitely complete category of which every change-of-base functor in the fibration of points is \emph{geometric}. Recall~\cite{Johnstone:Elephant} that such is a functor which preserves finite limits and jointly extremally epimorphic (small) families of arrows. This extends the definition of a \emph{logos} or \emph{geometric category} from topos theory~\cite{Alligators,Johnstone:Elephant} where the same is asked of the change-of-base functors in the basic fibration of a regular category.

We give examples and study the stability of the condition, which is a strengthening of the concept of an \emph{algebraically coherent} category~\cite{acc}. There, the change-of-base functors in the fibration of points are merely \emph{coherent}---jointly extremally epimorphic \emph{pairs} of arrows are preserved. Trivially, all categorical-algebraic consequences of the latter condition remain valid here: see Theorem~\ref{Theorem: semi-abelian consequences of existence of action cores} for an overview.

In any variety of universal algebras, the two concepts agree (Theorem~\ref{Theorem Coherent iff Geometric for Varieties}). Here we may show that \emph{normal cores} always exist (Theorem~\ref{Thm directed colimits}). In the abstract setting where our main interest lies---of a well-powered semi-abelian category with joins---they are generally different. It is here that we study the more general \emph{action cores} and \emph{split extension cores}, showing that their existence is equivalent to the category being an algebraic logos (Theorem~\ref{Main theorem}). We further prove (Theorem~\ref{Theorem (FWACC)}) that any such algebraic logos is a \emph{fibrewise algebraically cartesian closed} category; as a consequence, it has centralisers.

\subsection{Normal cores.}
The \defn{normal core} of a subgroup $S\leq X$ is the largest normal subgroup $N$ of~$X$ contained in $S$, which may be obtained as the join of all normal subgroups of $X$ smaller than $S$. This concept from group theory immediately extends to the general context of semi-abelian categories~\cite{Janelidze-Marki-Tholen}, and so does the construction of normal cores, as soon as the category in question is well-powered and joins of normal subobjects are normal. As we shall see in Section~\ref{Section Normal Cores}, both conditions are automatically satisfied in semi-abelian varieties of algebras; interestingly, in an abstract semi-abelian category, \emph{binary} joins of normal subobjects are always normal, whereas the condition on arbitrary collections happens to be equivalent to the existence of normal cores (Proposition~\ref{Proposition Normal cores vs joins}).

\subsection{Action cores.}
By definition, a normal subgroup $N$ of a group $X$ is a subgroup of $X$ which is closed under the conjugation action of $X$ on itself. A major goal of this article is to study what happens when we extend the definition of a normal core from the conjugation action to arbitrary actions: to define \emph{action cores}\footnote{This has nothing to do with the concept introduced in \cite{Actions} which carries the same name.} in semi-abelian categories, and to determine under which conditions on the category, action cores always exist. (See Section~\ref{Section Action Cores} where this is done in detail.) In particular, of course, any semi-abelian category with action cores has normal cores, but a major difference between the two is that a semi-abelian variety of algebras may fail to have action cores in general. Indeed, Proposition~\ref{Prop Action cores vs alg logos} tells us that a well-powered semi-abelian category with joins has action cores if and only if it is an algebraic logos.

\subsection{Split extension cores.}
The proof of Proposition~\ref{Prop Action cores vs alg logos} depends on the equivalence between actions and split extensions in the semi-abelian context, which allows us to reformulate the concept of an action core in terms of split extensions: thus we obtain the definition of a \emph{split extension core}, which makes sense in a wider setting and is more directly related to the concept of a geometric functor.

In a pointed category with finite limits $\C$, given a split extension $(\alpha,\beta,\kappa)$ as in the bottom of the diagram
\begin{equation*}
	\vcenter{
	\xymatrix{
	\bar X \ar@{-->}[r]^{\bar \kappa}\ar@{-->}[d]_{u} & \bar A \ar@{-->}@<0.5ex>[r]^{\bar \alpha}\ar@{-->}[dd]^{v} & B\ar@{-->}@<0.5ex>[l]^{\bar\beta}\ar@{=}[dd]\\
	S \ar@{{ >}->}[d]_{s}& & \\
	X \ar[r]_{\kappa} & A \ar@<0.5ex>[r]^{\alpha} & B\ar@<0.5ex>[l]^{\beta}
	}
	}
\end{equation*}
and a monomorphism $s\colon{S\to X}$, a \defn{split extension core} of $s$ relative to $(\alpha,\beta,\kappa)$ is a universal (terminal) lifting as displayed with dashed arrows above. When for each split extension $(\alpha, \beta, \kappa)$ and each $S\leq X$ a split extension core exists, we say that \defn{$\C$ has split extension cores}.

Alternatively, split extension cores may be viewed in terms of adjoint functors; this is an instance of a more general concept (of a \emph{subterminal partial right adjoint}) which is still under development: see Proposition~\ref{Proposition Action core via adjoint}.

\subsection{Structure of the text.}
We start in Section~\ref{Section Geom Functors} with a review of geometric functors and their properties. This leads to the definition of an algebraic logos in Section~\ref{Section Algebraic Logoi}, where we study the concept's main properties, equivalent formulations, etc. In Section~\ref{Section Split Extension Cores} we introduce split extension cores and relate them to algebraic logoi. Section~\ref{Section Action Cores} reformulates everything in terms of internal actions. In Section~\ref{Section Normal Cores} we specialise to normal cores, with the aim of proving that they exist in any semi-abelian variety of algebras. Section~\ref{Section Equivalent conditions} gives an overview of the relations between all of the concepts that appear in this article, Section~\ref{Section Examples} is devoted to examples and stability conditions (which may be used to construct further examples), and Section~\ref{Section Internal Groups} focuses on the case of internal groups in a geometric category.

\section{Geometric functors, geometric categories}\label{Section Geom Functors}
Recall that a (set-indexed) family of arrows $(u_i\colon U_i\to X)_{i\in I}$ is called \defn{jointly extremally epimorphic} when any monomorphism $u\colon {U\to X}$ through which the family $({u_i})_{i\in I}$ factors (as a collection of arrows $(\bar u_i\colon {U_i\to U})_{i\in I}$ so that $u\comp \bar u_i=u_i$) is an isomorphism.

\begin{definition}\label{Definition Geometric Functor}
	A functor between categories with finite limits is said to be \defn{geometric} when it preserves finite limits and jointly extremally epimorphic (small) families of arrows.
\end{definition}

The next proposition shows that for regular categories, this definition coincides with the one implicit in Section~A.1.4 of~\cite{Johnstone:Elephant}.

\begin{definition}
	Given an object $X$ and a small family of subobjects $(U_i\leq X)_{i\in I}$, a~subobject $U\leq X$ is called the \defn{join} of the $U_i$ in $X$ if and only if $U$ is the smallest subobject of $X$ such that $U_i\leq U$ for all $i\in I$. We write $U=\bigvee_{i\in I}U_i\leq X$.
\end{definition}

When $u_i\colon U_i\to X$ represents $U_i\leq X$ and $u\colon U\to X$ represents $U\leq X$, this means that
\begin{enumerate}
	\item each $u_i$ factors through $u$ via a monomorphism $\bar u_i\colon {U_i\to U}$ such that $u\comp \bar u_i=u_i$ and
	\item the family $(\bar u_i)_{i\in I}$ is jointly extremally epimorphic.
\end{enumerate}
In particular, a family of monomorphisms $(u_i\colon U_i\to X)_{i\in I}$ is jointly extremally epimorphic if and only if (considered as a family of subobjects $(U_i\leq X)_{i\in I}$) its join is $X$. 

\begin{definition}
	A finitely complete category \defn{has (small) joins} if for each small family of subobjects of a given object the join exists.
\end{definition}

Two well-known criteria for the existence of joins are relevant to us: any regular category with sums has small joins; so does any complete well-powered category. In the first case, we may consider a small family of subobjects $(U_i\leq X)_{i\in I}$ together with a chosen collection $(u_i\colon U_i\to X)_{i\in I}$ of representing monomorphisms. Then the image of the induced arrow $\lgroup u_i\rgroup_{i\in I}\colon \coprod_{i\in I} U_i \to X$ represents the join $\bigvee_{i\in I}U_i\leq X$. In the second case, any family of distinct subobjects $(U_i\leq X)_{i\in I}$ is necessarily small: $I$ is a set. Here we may use that a complete meet-semilattice is necessarily a complete lattice. We may now say that such a category has \emph{all} joins. (We no longer have to worry about the size of the indexing class~$I$, and may drop the word ``small''.)

Proposition~\ref{Proposition geom vs coh} expresses the relationship between geometric functors and the weaker notion of a coherent functor. Recall that a functor is \defn{coherent} when it preserves finite limits and jointly extremally epimorphic \emph{pairs} of arrows. It is easily seen---a detailed explanation is given in~\cite{acc}---that any coherent functor between regular categories is \defn{regular}, which means that it preserves finite limits and regular epimorphisms. In particular, it preserves image factorisations.

\begin{proposition}\label{Proposition joins}
	A regular functor between regular categories is geometric if and only if it preserves small joins.
\end{proposition}
\begin{proof}
	A family of arrows $(u_i\colon U_i\to X)_{i\in I}$ in a regular category is jointly extremally epimorphic if and only if the family of images $(\im(u_i)\colon \Im(u_i)\to X)_{i\in I}$ is jointly extremally epimorphic. Furthermore, any geometric functor preserves (regular epi, mono)-factorisations. The result follows.
\end{proof}

Recall that a \defn{directed colimit} is a colimit of a \defn{directed diagram}: a preorder (i.e., a relation which is reflexive and transitive) in which every finite subset has an upper bound. We notice that the join of a directed family $(U_i\leq X)_{i\in I}$ of subobjects of an object~$X$ may be obtained through a directed colimit: it is represented by the image of the arrow $\colim_{i\in I}U_i\to X$ induced by any choice of representing monomorphisms. (When coproducts exist, the colimit $\colim_{i\in I}U_i$ is a quotient of the coproduct $\coprod_{i\in I}U_i$, so that we may proceed as explained above.) If now $\C$ and~$\D$ are regular categories with small joins and directed colimits, then it follows that joins of directed families of subobjects in $\C$ are preserved by any functor $F\colon {\C\to \D}$ which is regular and preserves directed colimits. We obtain the following result, which is probably known. 

\begin{proposition}\label{Proposition geom vs coh}
	Suppose $\C$ and $\D$ are regular categories with small joins and directed colimits. If $F\colon {\C\to \D}$ is a coherent functor which preserves directed colimits, then $F$ is a geometric functor.
\end{proposition}
\begin{proof}
	We consider a set-indexed family $(U_i\leq X)_{i\in I}$ of subobjects of an object~$X$ of $\C$, and prove that its join is preserved by the functor $F$:
	\[
		F\bigl(\bigjoin_{i\in I}U_i\bigr)=\bigjoin_{i\in I}F(U_i)\leq F(X).
	\]
	For this, we consider a particular construction of that join by means of a directed colimit. We write $\Pf(I)=\{J\subseteq I\mid \text{$J$ finite}\}$ for the set of all finite subsets of~$I$. Then $(m_J\colon \bigjoin_{i\in J}U_i\to X)_{J\in \Pf(I)}$ represents a directed family of subobjects of~$X$, whose join, represented by the image of the arrow
	\[
		\lgroup m_J\rgroup_J\colon \colim_{J\in \Pf(I)}\bigjoin_{i\in J}U_i\to X,
	\]
	is precisely $\bigjoin_{i\in I}U_i\leq X$. This construction of $\bigjoin_{i\in I}U_i$ is indeed preserved by the functor $F$, because $F$ preserves finite joins (by coherence), directed colimits (by assumption), and image factorisations (by regularity: a coherent functor is always regular).
\end{proof}

\begin{definition}\label{Definition Geometric Category}
	A regular category with small joins $\C$ is \defn{geometric} in the sense of~\cite{Johnstone:Elephant} (and called a \defn{logos} in~\cite{Alligators}) if and only if, for any morphism $f\colon{X\to Y}$ in~$\C$, the change-of-base functor $f^{*}\colon (\C\downarrow Y)\to (\C\downarrow X)$ is geometric.
\end{definition}

\begin{remark}\label{Remark Logos}
	It is well known that in a regular context, the change-of-base functors $f^{*}\colon (\C\downarrow Y)\to (\C\downarrow X)$ are geometric if and only if so are the pullback functors $f^{*}\colon \Sub{\C}{Y}\to \Sub{\C}{X}$, and that this latter condition is equivalent to the pullback functors $f^{*}\colon \Sub{\C}{Y}\to \Sub{\C}{X}$ being left adjoints~\cite[1.7]{Alligators}. This need, however, not imply that the functors $f^{*}\colon (\C\downarrow Y)\to (\C\downarrow X)$ are left adjoints, which would make the category \emph{locally cartesian closed}---see the remark following Lemma~A.1.5.13 in~\cite{Johnstone:Elephant}.
\end{remark}

Just like for coherence (as remarked in~\cite{acc}), there are no non-trivial geometric semi-abelian categories. As in other cases~\cite{Beck,Bourn1991,GrayPhD,acc}, in Definition~\ref{Definition Geometric Category} we replace the basic fibration by the fibration of points. We thus find the concept of an \emph{algebraic logos}.

\section{Algebraic logoi}\label{Section Algebraic Logoi}
Let $\C$ be a category and $B$ an object of $\C$. We write $\Pt{\C}{B}$ for the \defn{category of points over $B$ in $\C$}. An object in $\Pt{\C}{B}$ is a pair $(\alpha\colon A\to B,\beta\colon B\to A)$ where $\alpha\comp\beta=1_B$, and a morphism 
\[
	f\colon(\alpha\colon A\to B,\beta\colon B\to A)\to (\bar\alpha\colon \bar A\to B,\bar\beta\colon B\to \bar A)
\]
is an $f\colon A\to \bar A$ such that $\bar\alpha\comp f=\alpha$ and $f\comp \beta=\bar\beta$.

\begin{definition}\label{Definition algebraic logos}
	A category~$\C$ with finite limits is called an \defn{algebraic logos} (or, rather awkwardly, an \defn{algebraically geometric category}) if and only if for each morphism $f\colon{X\to Y}$ in~$\C$, the change-of-base functor
	\[
		f^{*}\colon {\Pt{\C}{Y}\to \Pt{\C}{X}}
	\]
	determined by pulling back along $f$ is geometric.
\end{definition}

\begin{theorem}\label{Theorem Coherent iff Geometric for Varieties}
	A quasi-variety of algebras is an algebraic logos if and only if it is an algebraically coherent category.
\end{theorem}
\begin{proof}
	This follows immediately from Proposition~\ref{Proposition geom vs coh} and the fact that in a quasi-variety of algebras, directed colimits commute with finite limits.
\end{proof}

\begin{remark}
	We know of no interpretation of this condition in terms of subobjects precisely as in Remark~\ref{Remark Logos}. On the other hand, in both cases, ``being geometric'' can be phrased in terms of preservation of jointly extremally epimorphic families. Furthermore, Theorem~\ref{Theorem intersections} below explains that something slightly different does work. 
\end{remark}

\begin{remark}
	Recall~\cite{Tohoku} that a cocomplete abelian category satisfies the axiom (AB5) when in it, a directed colimit of short exact sequences is still a short exact sequence. In the abelian setting, this is well known to be equivalent to the condition that kernels commute with directed colimits, which amounts to saying that finite limits and directed colimits commute. It is precisely this latter, stronger variant of the (AB5) condition, which in the context of a cocomplete category is what we need to make the proof of Theorem~\ref{Theorem Coherent iff Geometric for Varieties} work: for cocomplete categories that satisfy it, ``geometric = coherent''.
\end{remark}

\begin{lemma}\label{Lemma Subobject}
	The forgetful functors $\Pt{\C}{X}\to (\C\downarrow X)$ and $(\C\downarrow X)\to \C$ preserve and reflect jointly extremally epimorphic families of arrows.
\end{lemma}
\begin{proof}
	The given functors preserve kernel pairs and are conservative, hence they preserve and reflect monomorphisms, which immediately implies the claim.
\end{proof}

\subsection{Stability properties.}
Using essentially the same proofs as in~\cite{acc}, we recover some of the stability properties of algebraically coherent categories.

\begin{proposition}\label{Prop Slice-coslice}
	If a category $\C$ is an algebraic logos, then, for each $X$ in~$\C$, the categories $(\C\downarrow X)$ and $(X\downarrow\C)$ are also algebraic logoi.
\end{proposition}
\begin{proof}
	For each morphism 
	\[
		\xymatrix{
			Y \ar[dr]_\alpha \ar[rr]^f & & Z \ar[dl]^\beta \\
			& X
		}
	\]
	in the slice category $(\C \downarrow X)$, there are isomorphisms of categories (the horizontal arrows below) which make the diagram
	\[
		\xymatrix{
		\Pt{\C\downarrow X}{(Z,\beta)} \ar[r]^-\cong \ar[d]_{(f \downarrow X)^{*}} & \Pt{\C}{Z} \ar[d]^{f^*} \\
		\Pt{\C\downarrow X}{(Y,\alpha)} \ar[r]_-\cong & \Pt{\C}{Y}
		}
	\]
	commute. It follows that $(f\downarrow X)^{*}$ is geometric whenever $f^*$ is. A similar argument holds for the coslice category $(X \downarrow \C)$.
\end{proof}

\begin{corollary}\label{Corollary Fibres}
	If a category $\C$ is an algebraic logos, then any fibre $\Pt{\C}{X}$ is an algebraic logos as well.
\end{corollary}
\begin{proof}
	This is a consequence of Proposition~\ref{Prop Slice-coslice} because $\Pt{\C}{X}$ may be seen as $((X\downarrow 1_X)\downarrow(\C\downarrow X))$.
\end{proof}

\begin{proposition}\label{Proposition Functor Categories}
	If $\C$ is an algebraic logos, then so is any category of diagrams in~$\C$. In particular, such are the categories $\Arr(\C)$ of arrows, $\PT{\C}$ of points, and $\RG(\C)$ of reflexive graphs in $\C$.
\end{proposition}
\begin{proof}
	This is an immediate consequence of the fact that in a category of diagrams, limits and colimits are computed pointwise. 
\end{proof}

\begin{proposition}\label{Proposition Reflective Subcat}
	If $\B$ is a full subcategory of an algebraic logos~$\C$ closed under finite products and subobjects (and hence all finite limits), then $\B$ is an algebraic logos.
	In particular, any full replete (extremal epi)-reflective subcategory of an algebraic logos is an algebraic logos.
\end{proposition}
\begin{proof}
	Let us recall the argument that a full replete (extremal epi)-reflective subcategory $\B$ of a category $\C$ is always closed under subobjects. Let $m\colon C\to B$ be a monomorphism in $\C$ whose codomain $B$ lies in $\B$. Apply the reflector $F\colon \C\to \B$ to $m$, which yields a morphism $F(m)$ in $\B$ making the square 
	\[
		\xymatrix{C \ar@{->>}[d]_{\eta_C} \ar@{ >->}[r]^-{m} & B \ar[d]^-{\eta_B}_{\cong}\\
		F(C) \ar[r]_-{F(m)} & F(B)}
	\]
	in $\C$ commute. Since $m$ is a monomorphism, so is the composite $\eta_B\comp m$. Hence $\eta_C$ is both a monomorphism and a regular epimorphism, which makes it an isomorphism, so that $C$ and $m$ lie in $\B$. Finally, the inclusion of $\B$ into $\C$ reflects monomorphisms, hence $m$ is a monomorphism in $\B$.
\end{proof}

\subsection{The protomodular case.}
We recall that a category $\C$ is called \defn{protomodular} in the sense of Bourn~\cite{Bourn1991} if it has pullbacks of split epimorphisms along any morphism and all the change-of-base functors of the fibration of pointed objects $\cod\colon {\PT{\C}\to \C}$ are \defn{conservative}, which means that they reflect isomorphisms. 

See also~\cite{Borceux-Bourn} for a detailed account of this notion, where a category which is pointed, regular and protomodular is called a \defn{homological} category\footnote{This is not to be confused with the homological categories of~\cite{Grandis-HA2}.}. In a homological category, classical diagram lemmas of homological algebra such as the $3\times 3$-Lemma and the Short Five Lemma are valid. A \defn{semi-abelian} category~\cite{Janelidze-Marki-Tholen}\footnote{This notion is different from Ra\"ikov's~\cite{Raikov:Semi-Abelian}.} is a homological category which is Barr exact and admits binary coproducts. A~long list of examples is given in~Section~\ref{Section Examples}; let us here just remark that typical semi-abelian categories are the categories of groups, rings, Lie algebras, compact Hausdorff groups, etc., as well as all abelian categories, while the category of topological groups is homological but not semi-abelian. 

Using Lemma~\ref{Lemma Subobject} we will prove that when $\C$ is a pointed protomodular category, the condition in Definition~\ref{Definition algebraic logos} can be expressed in terms of kernel functors $\Ker\colon \Pt{\C}{X}\to \C$ (which are precisely the change-of-base functors along initial morphisms $!_{X}\colon{0\to X}$) alone.

\begin{lemma}\label{Lemma Reg Protomodular}
	Let $F\colon \C\to \D$ be a functor. If $F$ is conservative and preserves monomorphisms then it reflects jointly extremally epimorphic families.
\end{lemma}
\begin{proof}
	Suppose that $(u_i\colon U_i\to X)_{i\in I}$ is a family in $\C$ whose image $(F(u_i))_{i\in I}$ in~$\D$ is a jointly extremally epimorphic family. This means that the image $F(u)$ of any monomorphism $u\colon {U\to X}$ in $\C$ through which the family $(u_i)_{i\in I}$ factors is an isomorphism. The result now follows from the assumption that $F$ reflects isomorphisms.
\end{proof}

\begin{proposition}\label{Proposition Regular Protomodular}
	If $\C$ is a protomodular category, then the change-of-base functors of the fibration of points reflect jointly extremally epimorphic families.\noproof
\end{proposition}

\begin{proposition}\label{Proposition Characterisation kernel functor}
	A protomodular category $\C$ with an initial object is an algebraic logos if and only if
	the change-of-base functors along each morphism from the initial object are geometric.
	In particular a pointed protomodular category is an algebraic logos if and only if the kernel functors $\Ker\colon \Pt{\C}{X}\to \C$ are geometric.
\end{proposition}
\begin{proof}
	Since by Proposition~\ref{Proposition Regular Protomodular} every change-of-base functor reflects jointly extremally epimorphic families, the non-trivial implication follows from inspecting the commutative triangle
	\[
		\xymatrix@!0@R=3em@C=4em
		{
		\Pt{\C}{Y} \ar[rr]^-{f^*} \ar[rd]_-{!_Y^*} && \Pt{\C}{X} \ar[ld]^-{!_X^*}\\
		& \Pt{\C}{0}
		}
	\]
	where $f\colon {X\to Y}$ is an arbitrary morphism in $\C$ and $0$ is the initial object in $\C$.
\end{proof}

This leads to the following interpretation in terms of subobjects alone, valid in any semi-abelian category. Here, unequivocally, a normal subobject is a subobject represented by the kernel of some morphism.

\begin{theorem}\label{Theorem intersections}
	A semi-abelian category is {\rm (1)} algebraically coherent or {\rm (2)} an algebraic logos, if and only if for each object $X$, each normal subobject $K\normal X$ and each subobject $B\leq X$ such that $K\meet B=0$ and $K\join B=X$, respectively:
	\begin{enumerate}
		\item all $U$, $V\leq X$ that contain $B$ satisfy $K\meet(U\join V)=(K\meet U)\join (K\meet V)$,
		\item for every small family of subobjects $(B\leq U_i\leq X)_{i\in I}$ which admits a join, we have
		      \[
			      K\meet\bigjoin_{i\in I} U_i=\bigjoin_{i\in I}(K\meet U_i).
		      \]
	\end{enumerate}
\end{theorem}
\begin{proof}
	A split extension of $B$ by $K$ is the same thing as an object $X$, together with a normal subobject $K\normal X$ and a subobject $B\leq X$ such that $K\meet B=0$ and $K\join B=X$. Indeed, if 
	\begin{equation*}
		\vcenter{
			\xymatrix{
				K \ar[r]_{\kappa} & X \ar@<0.5ex>[r]^{\alpha} & B\ar@<0.5ex>[l]^{\beta}
			}
		}
	\end{equation*}
	is a split extension, then $K\meet B=0$ is easy to see, while $K\join B=X$ is equivalent to protomodularity in any pointed regular category with binary joins. Conversely, given such $K\normal X$ and $B\leq X$, the canonical comparison $B\to X/K$ is a monomorphism because $K\meet B=0$ and an extremal epimorphism because $K\join B=X$. 
	
	Both results now follow, because a subobject of $(\alpha,\beta)$ in $\Pt{\C}{B}$ is the same thing as a subobject $U$ of $X$ that contains $B$. Furthermore, the kernel of its induced split epimorphism is $K\meet U$. Note that in (2) we are not assuming that every small family of subobjects admits a join; but when it does, the corresponding formula must hold, in accordance with Proposition~\ref{Proposition joins}. 
\end{proof}

\section{Split extension cores}\label{Section Split Extension Cores}
We give the definition of a split extension core and explain that any homological quasi-variety (or, in particular, any semi-abelian variety) has split extension cores if and only if it is an algebraic logos.

\begin{definition}\label{Definition Split Extension Core}
	In a pointed category with finite limits $\C$, given a split extension $(\alpha,\beta,\kappa)$ as in the bottom of the diagram
	\begin{equation}\label{Equation Explicit split ext core}
		\vcenter{
		\xymatrix{
		\bar X \ar@{-->}[r]^{\bar \kappa}\ar@{-->}[d]_{u} & \bar A \ar@{-->}@<0.5ex>[r]^{\bar \alpha}\ar@{-->}[dd]^{v} & B\ar@{-->}@<0.5ex>[l]^{\bar\beta}\ar@{=}[dd]\\
		S \ar@{{ >}->}[d]_{s}& & \\
		X \ar[r]_{\kappa} & A \ar@<0.5ex>[r]^{\alpha} & B\ar@<0.5ex>[l]^{\beta}
		}
		}
	\end{equation}
	and a monomorphism $s\colon{S\to X}$, a \defn{split extension core} of $s$ relative to $(\alpha,\beta,\kappa)$ is a universal (terminal) lifting as displayed with dashed arrows above.
	
	When for each split extension $(\alpha, \beta, \kappa)$ and each $S\leq X$ a split extension core exists, we say that \defn{$\C$ has split extension cores}.
\end{definition}

\begin{lemma}\label{Lemma mono}
	The morphisms $u$ and $v$ in~\eqref{Equation Explicit split ext core} are monomorphisms.
\end{lemma}
\begin{proof}
	Write $(R,\pi_1,\pi_2,e)$ for the kernel pair of $v$. Then $(\bar\alpha\comp\pi_1=\bar\alpha\comp\pi_2,e\comp\bar\beta)$ is a point over $B$, which together with $\pi_1$, $\pi_2$ and $e$ forms the kernel pair of $v$ in the category $\Pt{\C}{B}$. Its restriction $(\bar R, \bar \pi_1,\bar \pi_2,\bar e)$ to kernels is the kernel pair of $s\comp u$, which is also the kernel pair of $u$ because $s$ is a monomorphism. By universality of the split extension core $(\bar\alpha,\bar\beta,\bar\kappa)$, the morphisms $\pi_1$ and $\pi_2$ are equal, as well as $\bar \pi_1$ and $\bar \pi_2$, so that $u$ and $v$ are a monomorphisms. 
\end{proof}

\begin{theorem}\label{Thm Split Ext core vs alg logos}
	Let $\C$ be a pointed protomodular category. 
	\begin{enumerate}
		\item If $\C$ has split extension cores, then it is an algebraic logos.
		\item When $\C$ is well-powered with joins of subobjects, the conditions in {\rm (1)} are equivalent.
	\end{enumerate}
	For instance, in {\rm (2)} the category $\C$ can be a well-powered homological category which is complete or has small sums.
\end{theorem}
\begin{proof}
	We first prove (1). By Proposition~\ref{Proposition Characterisation kernel functor}, it suffices to show that the kernel functors $\Ker\colon \Pt{\C}{X}\to \C$ are geometric. Take an extremally epimorphic family $(v_i)_{i\in I}$ of arrows with codomain $(\alpha,\beta)$ in $\Pt{\C}{B}$, and suppose their restriction to kernels factors through some monomorphism $s$. Then they must factor through the (monic) split extension core of $s$, which lifts to a monomorphism $m$ with codomain~$(\alpha,\beta)$ through which the family $(v_i)_{i\in I}$ factors. By assumption, $m$ is an isomorphism, hence so is $s$.
	
	For the proof of (2), consider a split extension $(\alpha,\beta,\kappa)$ and a monomorphism $s\colon{S\to X}$ as in~\eqref{Equation Explicit split ext core}. Take the collection of all monomorphisms $v_i\colon {(\alpha_i,\beta_i)\to (\alpha,\beta)}$ whose restriction to kernels factors through $s$ via some monomorphism $u_i$. By assumption, the join of the family $(v_i)_i$ forms a point over $B$ which restricts to the join of the family $(s\circ u_i)_i$, the needed split extension core of~$s$.
\end{proof}

Together with the fact that all quasi-varieties of algebras are well-powered and (co)com\-plete, Theorem~\ref{Theorem Coherent iff Geometric for Varieties} now implies:

\begin{corollary}\label{Cor Split Ext Core in Variety}
	For a homological quasi-variety of algebras $\V$, the following conditions are equivalent:
	\begin{tfae}
		\item $\V$ is algebraically coherent;
		\item $\V$ is an algebraic logos;
		\item $\V$ has split extension cores.
	\end{tfae}
	This holds, in particular, for any semi-abelian variety.\noproof
\end{corollary}

Recall that semi-abelian varieties of algebras were characterised in terms of operations and identities in the article~\cite{Bourn-Janelidze}. This characterisation was extended to homological quasi-varieties in~\cite{Roque}, which were previously characterised by other means in~\cite{Gran-Rosicky:Monadic}. Note that a semi-abelian quasi-variety is automatically a variety, because it is a Barr-exact category.

Split extension cores may be viewed in terms of adjoint functors.

\begin{proposition}\label{Proposition Action core via adjoint}
	Let $\C$ be a finitely complete category. For each point $(\alpha,\beta)$ as in~\eqref{Equation Explicit split ext core}, consider the lifting of the kernel functor $\Ker\colon \Pt{\C}{B}\to \C$ to a functor
	\[
		F\colon \Sub{\Pt{\C}{B}}{(\alpha,\beta)}\to \Sub{\C}{X}.
	\]
	If $\C$ has split extension cores, then for each point $(\alpha,\beta)$ this functor has a right adjoint. The converse is true as soon as $\C$ is regular.
\end{proposition}
\begin{proof}
	When $\C$ has split extension cores, the needed right adjoint $G$ sends a subobject ${S\leq X}$ to the split extension core of $S$ with respect to $(\alpha,\beta)$.
	
	Conversely, let $\C$ be a regular category and $G$ a right adjoint of $F$. Consider the diagram
	\[
		\xymatrix{
			Y \ar[rr]^{\iota} \ar[d]_{f} \ar@{->>}[dr]^{\bar e} & & C \ar@<0.5ex>[rr]^{\gamma} \ar[dd]|\hole_(.7){g} \ar@{->>}[dr]^e & & B \ar@<0.5ex>[ll]^{\delta} \ar@{=}[d] \\
			S \ar@{ >->}[d]_{s} & Y' \ar[rr]^(.3){\iota'} \ar@{ >->}[dl]^{\bar m} \ar@{-->}[l] & & C' \ar@<0.5ex>[r]^{\gamma'} \ar@{ >->}[dl]^m & B \ar@{=}[d] \ar@<0.5ex>[l]^{\delta'} \\
			X \ar[rr]_{\kappa} & & A \ar@<0.5ex>[rr]^{\alpha} & & B \ar@<0.5ex>[ll]^{\beta}
		}
	\]
	where $(s\comp f,g)$ is a morphism of split extensions over $B$ and $s$ is a monomorphism.
	
	The (regular epi, mono)-factorisation $(e,m)$ of $g$ in $\C$ yields the (regular epi, mono)-factorisation $((\bar e,e),(\bar m,m))$ of $(s\comp f,g)$, whose restriction $\bar m$ to kernels factors through $s$ because $\bar e$ is a strong epimorphism. By adjunction, $(\bar m,m)$ factors through $G(s)$ and so does $(s\comp f,g)$ by composition. Hence $G(s)$ produces the needed split extension core. 
\end{proof}

\section{Action cores and the functors \texorpdfstring{$B\flat(-)$}{}}\label{Section Action Cores}
In a semi-abelian category, via the equivalence between internal actions and split extensions, we may reformulate the concept of a split extension core in terms of actions. The resulting definition of an \emph{action core} makes sense in a wider context, but in a semi-abelian category the two concepts agree.

Let us recall~\cite{Bourn-Janelidze:Semidirect,BJK,Borceux-Semiab} that for a pointed finitely complete category $\C$ with binary coproducts and for each $B$ in $\C$ there is a monad $B\flat(-)$ on $\C$ defined as follows. For each object~$X$ in $\C$,
\begin{itemize}
	\item $B\flat X$ is defined to be the (object part of the) kernel $(B\flat X,\kappa_{B,X})$ of $\mus{1_{B}}{0}$, the unique morphism making the diagram
	      \begin{equation*}
		      \vcenter{
		      \xymatrix@!@C=8ex@R=0ex{
		      & X \ar[d]_-{\iota_2}\ar[dr]^{0} &\\
		      B\flat X \ar[r]^-{\kappa_{B,X}} & B+X \ar[r]|{\mus{1_{B}}{0}} & B\\
		      & B \ar[u]^-{\iota_1}\ar[ur]_{1_B}
		      }
		      }
	      \end{equation*}
	      in which $\iota_1$ and $\iota_2$ are the first and second coproduct inclusions, commute;
	\item $\eta_X\colon X\to B\flat X$, the component of the unit at $X$, is the unique morphism such that $\kappa_{B,X}\circ \eta_X = \iota_2$;
	\item $\mu_X$, the component of the multiplication of the monad at $X$, is the unique morphism making the diagram
	      \begin{equation*}
		      \vcenter{
		      \xymatrix{
		      B\flat(B\flat X) \ar[r]^-{\kappa_{B,B\flat X}}\ar[d]_{\mu_X} & B+(B\flat X)\ar[d]_{1_B + \kappa_{B,X}} \ar[r]^-{\mus{1_{B}}{0}} & B\ar@{=}[d]\\
		      B\flat X \ar[r]_-{\kappa_{B,X}} & B+X \ar[r]_-{\mus{1_{B}}{0}} & B
		      }
		      }
	      \end{equation*}
	      commute.
\end{itemize}
An \defn{(internal) $B$-action} $(X,\xi)$ on an object $X$ is a $B\flat(-)$-algebra structure $\xi\colon{B\flat X\to X}$ on $X$.

It turns out that if $\C$ is the category of groups, then to give an algebra $(X,\xi)$ over a monad $B\flat(-)$ is the same thing as to give a $B$-group structure: a group homomorphism $B\to \Aut(X)$ to the automorphism group of $X$. Via exponential exchange this decodes to a group $X$ with multiplication $m\colon X\times X\to X$ and unit $e\colon 1\to X$ together with a morphism $a\colon B\times X \to X$ making the diagrams
\begin{gather*}
	\xymatrix{
		B\times (B\times X) \ar[rr]^{\alpha}\ar[d]_{1_B\times a} && (B\times B) \times X\ar[d]^{m\times 1_X} \\
		B\times X \ar[r]_{a} & X & B\times X \ar[l]^{a} &
	}\\
	\vcenter{
	\xymatrix{
	B\times (X\times X) \ar[rr]^-{\langle 1_B\times \pi_1,1_B\times \pi_2\rangle}\ar[d]_{1_B\times m} && (B\times X)\times(B\times X) \ar[d]^-{a\times a}\\
	B\times X\ar[r]_-{a} & X & X\times X\ar[l]^-{m}
	}
	}
	\qquad
	\vcenter{\xymatrix@=3em{
	X \ar[r]^-{\langle e\circ!_X, 1_X\rangle}\ar@{=}[dr] & B\times X\ar[d]^-{a}\\
	& X
	}}
\end{gather*}
commute. Here $\alpha$ is the usual ``associativity'' isomorphism. 

This correspondence, which is explained in \cite{Bourn-Janelidze:Semidirect,BJK,BJK2}, can be seen as a consequence of the fact that for a semi-abelian category~$\C$, for each $B$ in $\C$ there is an equivalence of categories between the category $\Act_{B}(\C)=\C^{B\flat (-)}$ of algebras in $\C$ over the monad $B\flat(-)$, and the category $\Pt{\C}{B}$ with objects split epimorphisms in $\C$ with a chosen splitting and with codomain~$B$.

\begin{definition}
	Given an object $B$, a $B$-action $(X,\xi)$ and a subobject $S \leq X$, in~$\C$, a pointed finitely complete category with binary coproducts, the \defn{action core of $S$ relative to $X$} is the largest sub-$B$-action $(Y,\varphi)$ of $(X,\xi)$ such that $Y$ is a subobject of $S$.
\end{definition}

Note that in a semi-abelian category, the action core of $S$ is precisely the internal action corresponding to the split extension core of $S$ with respect to the split extension corresponding to $\xi$. Thus, via Theorem~\ref{Thm Split Ext core vs alg logos} we find:

\begin{proposition}\label{Prop Action cores vs alg logos}
	A well-powered semi-abelian category with joins has action cores if and only if it is an algebraic logos.\noproof
\end{proposition}

We note that also the concept of an algebraic logos itself may be recast in terms of the action monad $B\flat (-)$. As it turns out, the reasoning in~\cite{acc} can be copied in order to prove the following result.

\begin{theorem}\label{Theorem Bemol}
	Let $\C$ be a homological category with small coproducts. $\C$ is an algebraic logos if and only if for every $B$, the functor $B\flat(-)\colon {\C\to \C}$ preserves jointly extremally epimorphic families of arrows.\noproof
\end{theorem}

\section{Normal cores; joins of normal subobjects}\label{Section Normal Cores}
In this section we discuss normal subobjects; throughout, we work in a semi-abelian category $\C$. We prove that unlike action cores, normal cores exist in any semi-abelian variety. This is closely related to the fact that there, joins of normal subobjects are normal. 

\begin{definition}
	Given a subobject $S\leq X$ in $\C$, a normal subobject $N\normal X$ is called a \defn{normal core} of $S\leq X$ if and only if $N\leq S$, terminal amongst all $M\normal X$ such that $M\leq S$.
	
	We say that \defn{$\C$ has normal cores} when each subobject $S\leq X$ in $\C$ has a normal core.
\end{definition}

\begin{proposition}\label{Proposition Normal core reformulation}
	For a subobject $S\leq X$ in $\C$, a normal core is a universal object amongst those subobjects of $S$ which are normal both in $S$ and in $X$.
\end{proposition}
\begin{proof}
	This is immediate from the definition.
\end{proof}

\begin{proposition}
	If $S\leq X$ has a normal core $N\normal X$, then the induced quotients determine a pullback
	\[
		\xymatrix{S \ar@{ >->}[r]^-s \ar@{->>}[d]_{q} & X \ar@{->>}[d]^r \\
		S/N \ar[r]_-{\bar s} & X/N}
	\]
	where the bottom arrow $\bar s$ is a monomorphism with trivial normal core.
\end{proposition}
\begin{proof}
	The facts that the square is a pullback and $\bar s$ is a monomorphism are consequences of the protomodularity of $\C$: Lemma~4.2.5 in~\cite{Borceux-Bourn}, for instance. We need to prove that the normal core of $\bar s$ is trivial. Suppose $M\leq S/N$ with $M\normal X/N$. Then the inverse image $r^{-1}(M)\normal X$ factors through~$S$, because the square is a pullback. So it is contained in the normal core $N$ of $S$, which implies that~$M$ is trivial.
\end{proof}

\begin{lemma}\label{Lemma Clots}
	$N\leq X$ is normal in $X$ if and only if the conjugation split extension in the bottom line of the diagram
	\[
		\xymatrix{N \ar@{-->}[r] \ar@{{ >}->}[d] & Y \ar@{-->}@<.5ex>[r] \ar@{-->}[d] & X \ar@{-->}@<.5ex>[l] \ar@{=}[d]\\
		X \ar[r]_-{\langle1_X,0\rangle} & X\times X \ar@<.5ex>[r]^-{\pi_2} & X \ar@<.5ex>[l]^-{\langle1_X,1_X\rangle}}
	\]
	restricts to a split extension of $X$ by $N$ as in the diagram's top line.
\end{lemma}
\begin{proof}
	This amounts to the equivalence between kernels and clots which holds in semi-abelian categories, and which follows from the results of~\cite{Janelidze-Marki-Ursini2, MM-NC}.
\end{proof}

\begin{proposition}\label{Proposition Normal core vs. split extension core}
	Any semi-abelian category with split extension cores has normal cores.
\end{proposition}
\begin{proof}
	Via Lemma~\ref{Lemma Clots}, a normal core is an instance of a split extension core.
\end{proof}

The following extends Corollary~3.21 in~\cite{acc} from binary to arbitrary (small) joins. Here a key difference between algebraically coherent and algebraically geometric categories (between algebraic pre-logoi and algebraic logoi) appears: while in any semi-abelian category a join of two normal subobjects is again normal \cite{Borceux-Semiab,Huq}, the condition that \emph{arbitrary} joins of normal subobjects are normal does not come for free.

\begin{proposition}\label{Proposition Normal cores vs joins}
	A well-powered semi-abelian category $\C$ with joins has normal cores if and only if in $\C$ joins of normal subobjects are normal.
\end{proposition}
\begin{proof}
	If $S\leq X$, then its normal core is obtained as the join of the set of all subobjects of $S$ which are normal in $X$. Conversely, given a set of normal subobjects of $X$, the normal core $N$ of their join $S$ coincides with $S$ by the universality of $S$. Hence $S$ is normal in~$X$.
\end{proof}

\begin{example}
	We give an example of a semi-abelian category whose joins of normal subobjects need no longer be normal: we consider the category of abelian groups with an additional operation $\omega$ of arity $\aleph_0$ which satisfies $\omega(0,\dots,0,\dots)=0$.
	
	First we notice that a subalgebra $N\leq X$ is normal if and only if for all $\alpha\in X^{\N}$ and $\beta\in N^{\N}$ we have $\omega(\alpha+\beta)-\omega(\alpha)\in N$. Indeed, the equivalence relation $R$ on~$X$ defined by
	\[
		(x,y)\in R \quad \Leftrightarrow \quad y-x\in N
	\]
	is a congruence when $(\alpha,\gamma)\in R^{\N}$ implies that $(\omega(\alpha),\omega(\gamma))=(\omega(\alpha),\omega(\alpha+(\gamma-\alpha)))$ is in $R$, which means that $\omega(\alpha+(\gamma-\alpha))-\omega(\alpha)$ is in $N$.
	
	We consider the abelian group $X=\Z_2^{\N}\times \Z$, whose elements we denote $(x,\bar x)=((x_m,\bar x))_{m\in \N}$, equipped with the operation $\omega\colon X^\N\to X$ which takes
	\[
		\alpha=((\alpha_n,\bar\alpha_n))_{n\in \N}=((\alpha_{mn},\bar\alpha_n))_{m, n\in \N}
	\]
	and sends it to 
	\begin{align*}
		\begin{cases}
			(0,0) & \text{if $\{(\alpha_{n},\bar\alpha_n) \mid n\in \N\}$ is finite,} \\
			(0,0) & \text{if $\bar\alpha_n=0$ for all $n\in \N$,}                     \\
			(0,1) & \text{otherwise.}
		\end{cases}
	\end{align*}
	For each $i\in \N$, the restriction of $\omega$ to the subgroup 
	\[
		N_i\coloneq\{(x,0)\in X\mid m\geq i \Rightarrow x_m=0\}
	\]
	of $X$ is necessarily zero. Each $N_i$ is a normal subobject of $X$, because for all $\alpha\in X^{\N}$ and $\beta\in N_i^{\N}$, by the common bound on the factors of $\beta$, the sequences $\alpha$ and $\alpha+\beta$ have the same image under $\omega$. 
	
	On the other hand, the union $N\coloneq \bigcup_{i\in \N}N_i$ is a subalgebra of $X$ which need not be normal. To see this, we take $\alpha=(0,1)$ and $\beta=((\delta_n,0))_{n\in \N}$ where $\delta_{mn}$ is $1$ when $n=m$ and zero otherwise: then $\omega(\alpha+\beta)-\omega(\alpha)=(0,1)$, which is not in $N$. 
\end{example}

In a varietal context, the difference disappears: Theorem~\ref{Thm directed colimits} below states that in any semi-abelian variety of algebras, arbitrary joins of normal subobjects are normal.

\begin{lemma}\label{Lemma Cosmash directed colimit}
	In a semi-abelian category where directed colimits commute with finite limits, each co-smash product functor $(-)\cosmash X$ preserves directed colimits. As a consequence, the Higgins commutator $[-,X]$ preserves directed joins of subobjects of $X$.
\end{lemma}
\begin{proof}
	This a direct consequence of the definitions: for any object $N$, we have a short exact sequence
	\[
		\xymatrix{
		0 \ar[r] & N\cosmash X \ar[r]^-{\iota_{N,X}} & N+X \ar[r] & N\times X \ar[r] & 0,
		}
	\]
	where both $(-)+X$ and $(-)\times X$ preserve directed colimits, as do kernels; furthermore, if $N\leq X$ is represented by $n\colon {N\to X}$, then $[N,X]$ is the regular image of $\mus{n}{1_X}\comp \iota_{N,X}$ as explained in~\cite{MM-NC}. The join of the regular images of a directed family is the regular image of its directed colimit.
\end{proof}

\begin{theorem}\label{Thm directed colimits}
	In any well-powered semi-abelian category with joins where directed colimits commute with finite limits, normal cores exist.
	
	In particular, normal cores exist in all semi-abelian varieties of algebras, and joins of normal subobjects are normal.
\end{theorem}
\begin{proof}
	Via Proposition~\ref{Proposition Normal cores vs joins}, it suffices to prove that a join of normal subobjects is normal. Since in any semi-abelian category finite joins of normal subobjects are normal, it suffices to prove that joins of directed sets of normal subobjects are still normal, as in Proposition~\ref{Proposition geom vs coh}. We use the characterisation of normality of a subobject via Higgins commutators proved in~\cite{MM-NC} which says that $N\leq X$ is normal in $X$ if and only if $[N,X]\leq N$. Then we only need to show that
	\[
		\bigl[\bigjoin_{i\in I}N_i,X\bigr]=\bigjoin_{i\in I}[N_i,X]
	\]
	for any directed set $(N_i\normal X)_{i\in I}$. This is a consequence of the fact that under the given assumptions, Lemma~\ref{Lemma Cosmash directed colimit} applies.
\end{proof}

Note the huge difference in strength between existence of normal cores and existence of split extension cores: in a semi-abelian variety, the former condition is always true, whereas the latter has non-trivial consequences---some of which we discuss in the next section.

\section{Equivalent conditions; consequences}\label{Section Equivalent conditions}
Combining the foregoing results, we find:

\begin{theorem}\label{Main theorem}
	For a well-powered semi-abelian category $\C$ with sums, the following conditions are equivalent:
	\begin{tfae}
		\item $\C$ is an algebraic logos;
		\item $\C$ has split extension cores;
		\item $\C$ has action cores;
		\item for every $B$ in $\C$, the functor $B\flat(-)\colon {\C\to \C}$ is geometric.\noproof
	\end{tfae}
\end{theorem}

We now work towards Theorem~\ref{Theorem (FWACC)}, a refinement of Theorem~6.27 in~\cite{acc} saying that any well-powered semi-abelian algebraic logos with joins is \emph{fibrewise algebraically cartesian closed}. In particular, such algebraic logoi have centralisers.

\begin{lemma}\label{lemma:strong action=>pullback along split epi has right adjoint}
	Let $\C$ be a pointed protomodular category. If~$\C$ has split extension cores, then for each split epimorphism $p\colon E\to B$, the induced functor
	\[
		p^*\colon {\Pt{\C}{B}\to \Pt{\C}{E}}
	\]
	has a right adjoint.
\end{lemma}
\begin{proof}
	We will actually exhibit a right adjoint for the corresponding functor
	\[
		p^*\colon {\SplExt{\C}{B}\to \SplExt{\C}{E}},
	\]
	between the categories of (isomorphism classes of) split extensions over~$B$ and over~$E$ respectively. The result follows by the equivalence between split extensions and points.
	
	Suppose that $p\colon E\to B$ is a morphism in $\C$ with section $s$. Consider a split extension
	\[
		\xymatrix{
			Y \ar[r]^{\iota} & D \ar@<.5ex>[r]^{\gamma} & E \ar@<.5ex>[l]^{\delta}
		}
	\]
	and its image
	\[
		\xymatrix{
			Y \ar[r]^{\iota'} & D' \ar@<.5ex>[r]^{\gamma'} & E \ar@<.5ex>[l]^{\delta'}
		}
	\]
	under the change-of-base functor $(s\comp p)^*$. Consider the split extension core diagram of~$\langle 1_Y,1_Y\rangle$ with respect to the product of $(\gamma,\delta,\iota)$ and $(\gamma',\delta',\iota')$ in the fibre over $E$:
	\[
		\xymatrix@R=2em@C=5em{
		\bar Y \ar[r]^{\bar\iota} \ar[d]_{u} & \bar D \ar@<0.5ex>[r]^{\bar\gamma}\ar[dd]^{v} & E \ar@<0.5ex>[l]^{\bar\delta} \ar@{=}[dd] \\
		Y \ar@{{ >}->}[d]_{\langle 1_Y,1_Y\rangle} & & \\
		Y\times Y \ar[r]^-{\iota\times \iota'} & D\times_E D' \ar@<0.5ex>[r]^-{\gamma\circ \pi_1=\gamma'\circ \pi_2} & E \ar@<0.5ex>[l]^-{\langle \delta,\delta'\rangle}
		}
	\]
	It is clear that the assignment
	\[
		C\colon\SplExt{\C}{E}\to \SplExt{\C}{E}\colon(\gamma,\delta,\iota) \mapsto (\bar\gamma,\bar\delta,\bar\iota)
	\]
	is functorial. Let us show that the composite $s^*\comp C$ provides the desired right adjoint of $p^*$.
	
	We first prove that $(\bar\gamma,\bar\delta,\bar\iota)$ is isomorphic to its image
	\[
		\xymatrix{
			\bar Y \ar[r]^{\bar\iota'} & \bar D' \ar@<.5ex>[r]^{\bar\gamma'} & E \ar@<.5ex>[l]^{\bar\delta'}
		}
	\]
	under $(s\comp p)^*$. It suffices to consider the following diagram:
	\[
		\xymatrix@R=3em@C=5em{
		\bar Y \ar[r]^{\bar\iota} \ar@{=}[d] & \bar D \ar@{-->}[d]^t \ar@/_4ex/[dd]_(.3){\pi_1\comp v} \ar@<.5ex>[r]^{\bar\gamma} & E \ar@<.5ex>[l]^{\bar\delta} \ar@{=}[d] \\
		\bar Y \ar[r]_(.4){\bar\iota'} \ar@{{ >}->}[d]_{u} & \bar D' \ar[d]^{(s\comp p)^*(\pi_1\comp v)} \ar@<.5ex>[r]^{\bar\gamma'} & E \ar@<.5ex>[l]^{\bar\delta'} \ar@{=}[d] \\
		Y \ar[r]^{\iota} & D \ar@<.5ex>[r]^{\gamma} & E \ar@<.5ex>[l]^{\delta}
		}
	\]
	Since by Lemma~\ref{Lemma mono} $u$ is a monomorphism, so is $(s\comp p)^*(\pi_1\comp v)$ by protomodularity: the square $(s\comp p)^*(\pi_1\comp v)\comp\bar\iota'=\iota\comp u$ is a pullback, and pullbacks in a protomodular category reflect monomorphisms~\cite{Bourn1991}. On the other hand, the pair $(\bar\iota,\bar\delta)$ is jointly strongly epimorphic by protomodularity, so there is a unique $t\colon \bar D \to \bar D'$ such that
	\[
		(s\comp p)^*(\pi_1\comp v)\circ t =\pi_1\comp v,\qquad t\circ\bar\iota=\bar\iota',\qquad t\circ\bar\delta=\bar\delta'.
	\]
	In particular, the upper part of the diagram gives a morphism of split extensions. Hence, by protomodularity, $t$ is an isomorphism.
	
	The morphism $(u,\pi_1\comp v)\colon p^*s^*C(\gamma,\delta,\iota)=p^*s^*(\bar\gamma,\bar\delta,\bar\iota)\cong(\bar\gamma,\bar\delta,\bar\iota)\to(\gamma,\delta,\iota)$ is the component at $(\gamma,\delta,\iota)$ of the adjunction counit $\epsilon$. To see this, we first observe that, for each
	\[
		\xymatrix{
			X \ar[r]^{\kappa} & A \ar@<.5ex>[r]^{\alpha} & B \ar@<.5ex>[l]^{\beta}
		}
	\]
	in $\SplExt{\C}{B}$, $(s\comp p)^*p^*(\alpha,\beta,\kappa)\cong p^*(\alpha,\beta,\kappa)$, hence $Cp^*(\alpha,\beta,\kappa)\cong p^*(\alpha,\beta,\kappa)$. Let us now prove that $\epsilon$ enjoys the universal property of a counit. For $(\alpha,\beta,\kappa)$ in $\SplExt{\C}{B}$ as above, consider $(f,g)\colon {p^*(\alpha,\beta,\kappa)\to(\gamma,\delta,\iota)}$. Then we have 
	\[
		\xymatrix@C=4em{
		p^*(\alpha,\beta,\kappa)\cong (s\circ p)^*p^*(\alpha,\beta,\kappa) \ar[r]^-{(s\circ p)^*(f,g)} & (s\circ p)^*(\gamma,\delta,\iota)=(\gamma',\delta',\iota')
		}
	\]
	whose restriction to kernels is still $f$. This induces a unique morphism
	\[
		{p^*(\alpha,\beta,\kappa)\to C(\gamma,\delta,\iota)}=(\bar\gamma,\bar\delta,\bar\iota).
	\]
	Applying $s^*$ to it, recalling that $s^*p^*\cong 1_{\SplExt{\C}{B}}$ we obtain
	\[
		(\bar f,\bar g)\colon {(\alpha,\beta,\kappa)\cong s^*p^*(\alpha,\beta,\kappa)\to s^*C(\gamma,\delta,\iota)}
	\]
	such that $\epsilon_{(\gamma,\delta,\iota)}\circ p^*(\bar f,\bar g)=(f,g)$.
\end{proof}

Whenever a finitely complete category $\C$ satisfies the conclusion of Lemma~\ref{lemma:strong action=>pullback along split epi has right adjoint}---for each split epimorphism $p\colon E\to B$, the functor $p^*\colon {\Pt{\C}{B}\to \Pt{\C}{E}}$ has a right adjoint---it is called a \defn{fibrewise algebraically cartesian closed} category~\cite{GrayPhD,Gray2012,Bourn-Gray}. When $\C$ is semi-abelian, the conditions
\begin{tfae}
	\item $\C$ is fibrewise algebraically cartesian closed;
	\item for each split epimorphism $p\colon E\to B$, the functor $p^{*}\colon\act{E}\to \act{B}$ has a right adjoint;
	\item for each $B$, the category $\act{B}$ has centralisers
\end{tfae}
are equivalent, so that we obtain:

\begin{theorem}\label{Theorem (FWACC)}
	Every pointed protomodular category admitting split extension cores is
	fibrewise algebraically cartesian closed.
	
	In particular, let $\C$ be a well-powered semi-abelian category with joins. If $\C$ is an algebraic logos, then $\C$ is a fibrewise algebraically cartesian closed category. This is equivalent to saying that for each object $B$ in~$\C$, the category $\act{B}$ has centralisers. When this happens, the category $\C \simeq \act{0}$ itself has centralisers.\noproof
\end{theorem}

Note that, unlike Theorem~6.27 in~\cite{acc} in the algebraically coherent context, this does not depend on preservation of finite limits by directed colimits; instead, we require well-poweredness and the mere existence of (small) joins, which follows for instance when arbitrary (small) limits or colimits exist.

\begin{problem}
Find an example of a semi-abelian variety which is fibrewise algebraically cartesian closed but not algebraically coherent.
\end{problem}

Since any algebraic logos is algebraically coherent, trivially all the consequences of algebraic coherence are also consequences of any of the above equivalent conditions, and so combining various results from~\cite{acc} and the current paper we obtain:

\begin{theorem}\label{Theorem: semi-abelian consequences of existence of action cores}
	Let $\C$ be a semi-abelian category. If $\C$ is an algebraic logos, then
	\begin{itemize}
		\item all change-of-base functors of the fibration of points preserve Huq and Higgins commutators, normal closures and cokernels;
		\item the Huq, Higgins and Ursini/Smith~\cite{Mantovani:Ursini} commutators coincide for normal subobjects;
		\item $\C$ is \emph{peri-abelian}~\cite{Bourn-Peri}: change-of-base functors of the fibration of points preserve abelianisation; the \emph{universal central extension condition}~\cite{CVdL,GrayVdL1} holds: for central extensions $f\colon A\to B$ and $g\colon B\to C$ the composite $gf$ is central provided the abelianisation of $B$ is trivial;
		\item $\C$ is \emph{strongly protomodular}~\cite{B4,Rodelo:Moore}: change-of-base functors of the fibration of points reflect normal monomorphisms;
		\item the \emph{Three Subobjects Lemma} is valid for commutators of normal subobjects: if $K,L,M$ are normal subobjects of an object $X$, then
		      \[
			      [K,[L,M]]\leq [M,[K,L]]\join [L,[M,K]].
		      \]
	\end{itemize}
	If $\C$ has split extension cores---$\C$ may, for instance, be a well-powered algebraic logos with joins of subobjects---then
	\begin{itemize}
		\item $\C$ is fibrewise algebraically cartesian closed; for each object $B$ in $\C$, the category $\act{B}\simeq\Pt{\C}{B}$ has centralisers;
		\item $\C$ has normal cores; in $\C$, joins of normal subobjects are normal.\noproof
	\end{itemize}
\end{theorem}

\section{Examples and stability properties}\label{Section Examples}
It is well known~\cite[1.711]{Alligators} that \emph{any cocomplete locally cartesian closed category is geometric}. We find the following algebraic version of this classical result.
We recall from~\cite{Gray2012, Bourn-Gray} that a finitely complete category $\C$ is said to be \defn{locally algebraically cartesian closed} (satisfies condition \defn{(LACC)}) when, for every $f\colon X \to Y$ in \C, the change-of-base functor $f^*\colon \Pt{\C}{Y}\to \Pt{\C}{X}$ is a left adjoint.

\begin{theorem}\label{Theorem (LACC) implies (AL)}
	\begin{enumerate}
		\item Any \LACC\ category is an algebraic logos;
		\item any pointed \LACC\ category has split extension cores.
	\end{enumerate}
\end{theorem}
\begin{proof}
	For the first statement we may already notice that the change-of-base functors always preserve limits. Under \LACC, they are left adjoints, which implies that they preserve jointly extremally epimorphic sets of arrows as well: indeed, extremally epimorphic families coincide with strongly epimorphic families in any category with pullbacks, and the latter are well known and easily seen to be preserved by any left adjoint functor~\cite{Pumpluen}. 
	
	The second statement is a consequence of the fact that for each initial morphism $!_B\colon 0\to B$, the functor $!^*_B$ has a right adjoint $R_B\colon \Pt{\C}{0}\to \Pt{\C}{B}$. Indeed, given a situation as in Diagram~\ref{Equation Explicit split ext core}, the morphism $s\colon S\to !^*_B(\alpha,\beta)$ in $\C$ induces a morphism of split extensions $R_B(s)\colon R_B(S)\to R_B(!^*_B(\alpha,\beta))$, and the needed $v\colon (\bar\alpha,\bar\beta)\to (\alpha,\beta)$ is the pullback of this $R_B(s)$ along the component $\eta_{(\alpha,\beta)}\colon (\alpha,\beta)\to R_B(!^*_B(\alpha,\beta))$ at $(\alpha,\beta)$ of the unit of the adjunction. The universal property follows.
\end{proof}

\begin{example}
	The category of cocommutative Hopf algebras over a field is semi-abelian as explained in~\cite{GSV}. It is also locally algebraically cartesian closed by Proposition~5.3 in~\cite{Gray2012}, being the category of internal groups in the category of cocommutative coalgebras, which is cartesian closed as shown in~\cite[Theorem~5.3]{Barr-Coalgebras}. Finally, being locally presentable~\cite{Porst-CoLimits}, it is complete and cocomplete, so that small joins exist.
\end{example}

\begin{example}
	An abelian category, being a locally algebraically cartesian closed semi-abelian category, is always an algebraic logos.
\end{example}

\begin{problem}
We do not know---see also Theorem 4.11 in~\cite{acc}---whether the (homological) category of topological groups is an algebraic logos. 
\end{problem}

\begin{example}\label{Example Compact Hausdorff Algebras}
	On the other hand, the semi-abelian variety of compact Hausdorff groups is. This follows essentially from the criterion of Theorem~\ref{Theorem intersections}, together with the fact that a monomorphism in this category is the same thing as a subgroup inclusion, where the subgroup in question carries the topology induced by the codomain. Indeed, the category of compact Hausdorff spaces is a pre-topos (see~\cite[1.773]{Alligators} or \cite{CH}, for instance) so that every monomorphism in it is a regular monomorphism---i.e., an equaliser of some pair of morphisms; this implies that the topology on the domain of a monomorphism is induced by the topology on the codomain. As a consequence, the join $\bigjoin_{i\in I}U_i$ of any set of subobjects $(U_i\leq X)_{i\in I}$ is obtained as a join of subgroups, equipped with the topology induced by $X$. Hence the equality in the theorem does indeed hold. 
	
	Essentially the same argument works for arbitrary semi-abelian compact Hausdorff algebras; see~\cite{Borceux-Clementino} for further details about those. 
	
	The argument may also be modified to show that the category of profinite groups is a homological algebraically coherent category, because it may be seen as the category of internal groups in the category of Stone spaces, which is known to be a non-exact coherent category~\cite{CH}.
	
	This result may be extended in a different direction as well: it is an instance of a category of internal groups in a geometric category; in Section~\ref{Section Internal Groups} we treat those in detail, proving that such categories are always algebraic logoi.
\end{example}

Corollary~\ref{Cor Split Ext Core in Variety} provides another class of examples:

\begin{example}
	\emph{Orzech categories of interest}~\cite{Orzech} are algebraic logoi, since they are algebraically coherent categories~\cite{acc}. In particular, semi-abelian varieties of algebras which are algebraic logoi include: the categories of groups; non-unital (Boolean) rings; associative algebras, Lie algebras, Leibniz algebras, Poisson algebras over a commutative ring with unit; all \emph{varieties of groups} in the sense of~\cite{Neumann}.
\end{example}

As we shall see now, sometimes the converse holds. Let $\K$ be a field. A \defn{(non-associative) algebra $(A,\cdot)$ over $\K$} is a $\K$-vector space~$A$ equipped with a bilinear operation ${\cdot\colon A\times A\to A\colon (x,y)\mapsto x\cdot y}$. These form a variety of algebras $\Alg_\K$. We call a \defn{variety of $\K$-algebras} any subvariety of $\Alg_\K$.

\begin{corollary}
	Let $\V$ be a variety of $\K$-algebras over an infinite field $\K$. Then the following conditions are equivalent:
	\begin{tfae}
		\item $\V$ is an algebraic logos;
		\item $\V$ is an \emph{Orzech category of interest}~\cite{Orzech};
		\item $\V$ is action accessible~\cite{BJ07};
		\item there exist $\lambda_{1}$, \dots, $\lambda_{16}\in \K$ such that the equations
		\begin{align*}
			z(xy)=
			\lambda_{1}y(zx) & +\lambda_{2}x(yz)+
			\lambda_{3}y(xz)+\lambda_{4}x(zy)                         \\
			                 & +\lambda_{5}(zx)y+\lambda_{6}(yz)x+
			\lambda_{7}(xz)y+\lambda_{8}(zy)x                         \\
			(xy)z=
			\lambda_{9}y(zx) & +\lambda_{10}x(yz)+
			\lambda_{11}y(xz)+\lambda_{12}x(zy)                       \\
			                 & +\lambda_{13}(zx)y+\lambda_{14}(yz)x+
			\lambda_{15}(xz)y+\lambda_{16}(zy)x
		\end{align*}
		hold in $\V$;
		\item in $\V$, split extension cores exist.
	\end{tfae}
\end{corollary}
\begin{proof}
	This combines Theorem~\ref{Main theorem}, Theorem~\ref{Theorem Coherent iff Geometric for Varieties}, and results of~\cite{Edinburgh,GM-VdL2}.
\end{proof}

As for algebraically coherent categories, recycling the proofs given in~\cite{acc}, we find:

\begin{proposition}
	If $\C$ is a semi-abelian algebraic logos and $X$ is an object of $\C$, then the category $\Act_{X}(\C)=\C^{X\flat (-)}$ of $X$-actions in $\C$ is a semi-abelian algebraic logos.
\end{proposition}
\begin{proof}
	This follows from Corollary~\ref{Corollary Fibres} via the equivalence between $X$-actions and points over $X$.
\end{proof}

\begin{proposition}\label{Proposition XMod}
	If $\C$ is an exact Mal'tsev algebraic logos, then the category $\Cat(\C)$ of internal categories (=~internal groupoids) in $\C$ is an algebraic logos. As a consequence, the category $\Eq(\C)$ of (effective) equivalence relations in $\C$ is an algebraic logos.
	
	If, moreover, $\C$ is semi-abelian then, by equivalence, the categories $\PXMod(\C)$ and $\XMod(\C)$ of internal (pre)crossed modules in $\C$ are algebraic logoi.
\end{proposition}
\begin{proof}
	The first statement follows from Proposition~\ref{Proposition Functor Categories}. We now assume that $\C$ is exact Mal'tsev. Since the category of internal categories of~$\C$ is (regular epi)-reflective in $\RG(\C)$, we have that $\Cat(\C)$ is an algebraic logos by Proposition~\ref{Proposition Reflective Subcat}. In turn, following~\cite{Gran:Central-Extensions, Bourn-comprehensive}, we see that the category $\Eq(\C)$ is (regular epi)-reflective in $\Cat(\C)$. The final claim in the semi-abelian context now follows from the results of~\cite{Janelidze}.
\end{proof}

\begin{examples}
	Crossed modules (of groups, rings, Lie algebras, etc.); $n$-cat-groups, for all $n$~\cite{Loday}.
\end{examples}

\begin{proposition}
	If $\C$ is an exact Mal'tsev algebraic logos, then the full subcategory $\Ext(\C)$ of $\Arr(\C)$ determined by the extensions (=~regular epimorphisms), and the category $\CExt_{\B}(\C)$ of $\B$-central extensions~\cite{Janelidze-Kelly} in $\C$, for any Birkhoff subcategory~$\B$ of $\C$, are algebraic logoi.
\end{proposition}
\begin{proof}
	For the category $\Ext(\C)$, this is a consequence of Proposition~\ref{Proposition XMod}, by the equivalence $\Eq(\C)\simeq\Ext(\C)$ which holds in any Barr-exact category. The result for central extensions now follows from Proposition~\ref{Proposition Reflective Subcat}.
\end{proof}

\begin{examples}
	Inclusions of normal subgroups (considered as a full subcategory of~$\Arr(\Gp)$); central extensions of groups, Lie algebras, crossed modules; discrete fibrations of internal categories (considered as a full subcategory of $\Arr(\Cat(\C))$) in a semi-abelian algebraic logos $\C$~\cite[Theorem~3.2]{Gran:Central-Extensions}.
\end{examples}

\begin{proposition}
	Any sub-quasi-variety (in particular, any subvariety) of a variety which is an algebraic logos is an algebraic logos.
\end{proposition}
\begin{proof}
	Since any sub-quasi-variety is an (extremal epi)-reflective subcategory~\cite{Maltsev}, this follows from Proposition~\ref{Proposition Reflective Subcat}.
\end{proof}

\begin{examples}
	The varieties of $n$-nilpotent groups or $n$-solvable groups (for a fixed $n\geq 1$), rings, Lie algebras etc.; torsion-free (abelian) groups, reduced rings.
\end{examples}

\begin{remark}
	Remark 4.26 in~\cite{acc} may be strength\-ened, and thus leads to the idea (eventually to be developed in the future) that the category of monoids is a \defn{relative} algebraic logos, with respect to the \emph{fibration of Schreier points}~\cite{SchreierBook} rather than the fibration of points. In particular, all kernel functors $\Ker\colon{\SPt_{X}(\Mon)\to \Mon}$ are geometric, because they are left adjoints, as proved in~\cite{S-LACC}.
\end{remark}

\section{Internal groups in a geometric category}\label{Section Internal Groups}
We now extend Example~\ref{Example Compact Hausdorff Algebras} as announced there: we prove that the category of internal groups in a geometric category is algebraically geometric. Throughout we assume that $\C$ is a geometric category. Recall:

\begin{lemma}\label{dist}
	For each object $B$ in a geometric category $\C$, the functor 
	\[
		{B\times (-)}\colon \C\to \C\colon X\mapsto B\times X
	\]
	preserves joins.
\end{lemma}
\begin{proof}
	It suffices to note that $B\times (-)$ is the composite of the change-of-base functor
	$!_B^{*}\colon \C\to (\C\downarrow B)$ along $!_B\colon B\to 1$ and the forgetful functor $(\C\downarrow B)\to \C$. The former functor preserves joins by assumption, while the latter does so trivially; hence so does their composite.
\end{proof}

\begin{lemma}\label{joins_of_subgroup}
	Let $S$ and $T$ be internal subgroups of an internal group $X$ in $\C$.
	The join of $S$ and $T$ as internal groups is the join in $\C$ of the family
	of subobjects $J_0=S\join T$, $J_{n+1} = J_n\join m(J_n\times J_n)$ where
	$m(J_n\times J_n)$ is the image of subobject $J_n\times J_n$ of $X\times X$
	along the multiplication morphism $m\colon X\times X\to X$.
\end{lemma}
\begin{proof}
	Let the $J_n$ be defined as above and let us write $J$ for their join in $\C$. Trivially, since the unit $e\colon 1 \to X$ factors through $S$, it factors through $J_0$ and hence through~$J$. On the other hand, by Lemma~\ref{dist} and the fact that regular images preserve joins, we have:
	\begin{align*}
		m(J\times J) & = m\bigl(\bigl(\bigjoin_{k\in \N} J_k\bigr)\times \bigl(\bigjoin_{l \in \N} J_l\bigr)\bigr) = m\bigl(\bigjoin_{k\in \N} \bigjoin_{l\in \N} (J_k\times J_l)\bigr) \\
		             & =\bigjoin_{k\in \N} \bigjoin_{l\in \N} m(J_k\times J_l) \leq \bigjoin_{k\in \N} \bigjoin_{l\in \N} m(J_{\max\{k,l\}}\times J_{\max\{k,l\}})                      \\
		             & \leq \bigjoin_{k\in \N} \bigjoin_{l\in \N} J_{\max\{k,l\}+1} = J
	\end{align*}
	and so $m\colon X\times X\to X$ restricts to $J$. Writing $i\colon X\to X$ for the inverse (iso)morphism we note that $i(J_0)=i(S\join T)=i(S)\join i(T) = S\join T=J_0$
	and for each $n$ in $\N$ we have $i(J_{n+1})=i(J_n\join m(J_n\times J_n))=i(J_n)\join i m(J_n\join J_n)=i(J_n)\join m(i(J_n)\times i(J_n))$. Hence by induction it follows that $i(J_{n})=J_{n}$ for each $n$ in $\N$.
	Therefore
	\begin{align*}
		i(J) = i\bigl(\bigjoin_{k\in \N} J_k\bigr) = \bigjoin_{k \in \N} i(J_k) =\bigjoin_{k\in \N} J_k = J
	\end{align*}
	and so $i\colon X\to X$ restricts to $J$. It now follows that $J$ is a subgroup of $X$. 
	
	Suppose that $U$ is a subgroup of $X$ containing $S$ and $T$. We have that $J_0 = S\join T \leq U$. If $J_n \leq U$, then $m(J_n\times J_n)\leq m(U\times U)\leq U$ and hence $J_{n+1}=J_n\join m(J_n\times J_n) \leq U$. It follows by induction that $J_n\leq U$ for all $n\in \N$ and hence $J\leq U$ as desired.
\end{proof}

\begin{lemma}\label{coherence}
	If $B$ is an internal group in a geometric category and if $S$ and~$T$ are sub-$B$-groups of a $B$-group $X$, then the join of $S$ and $T$ as subgroups of $X$ admits a $B$-group structure making it a sub-$B$-group of $X$.
\end{lemma}
\begin{proof}
	We know that the join can be constructed as $\bigjoin_{k\in \N} J_k$ where the $J_k$ are as in Lemma~\ref{joins_of_subgroup}. We follow the notation of Section~\ref{Section Action Cores} and write $a\colon B\times X\to X$ for the action morphism of $X$. We note that $S$ and $T$ being sub-$B$-groups amounts to $a(B\times S)\leq S$ and $a(B\times T)\leq T$. We have
	$a(B\times J_0) = a(B\times (S\join T)) = a((B\times S) \join (B\times T))=a(B\times S)\join a(B\times T)\leq S\join T=J_0$. Now suppose that $a(B\times J_n) \leq J_n$ for some $n\in \N$. Then
	\begin{align*}
		a(B\times J_{n+1}) & = a(B\times (J_n \join m(J_n\times J_n))) 
		=a((B\times J_n) \join (B\times m(J_n\times J_n)))                         \\
		                   & =a(B\times J_n) \join a(B\times m(J_n\times J_n))     \\ 
		                   & \leq J_n \join m(a(B\times J_n)\times a(B\times J_n)) \\
		                   & \leq J_n \join m(J_n\times J_n) 
		= J_{n+1}.
	\end{align*}
	By induction, it follows that $a(B\times J_n) \leq J_n$ for all $n \in \N$.
\end{proof}

\begin{lemma}\label{directed}
	Suppose $B$ is an internal group in a geometric category $\C$. Suppose that $(S_j)_{j\in I}$ is a directed family of sub-$B$-groups of a $B$-group $X$. The join of $S_i$ as $B$-groups is the join $\bigjoin_{j \in I} S_j$ in $\C$, equipped with the induced structure.
\end{lemma}
\begin{proof}
	Let us write $S=\bigjoin_{j \in I} S_j$.
	Since for each $j\in I$ and $k\in \N$ there exists $l \in I$ such that $S_j\leq S_l$
	and $S_k\leq S_l$ it follows that $m(S_j\times S_k)\leq m(S_l\times S_l)\leq S_l\leq S$.
	Trivially, $e\colon 1\to X$ factors through $S$. On the other hand we have
	\begin{align*}
		m(S\times S) & =m\bigl(\bigjoin_{j\in I}S_i \times \bigjoin_{k\in I} S_j\bigr) = m\bigl(\bigjoin_{j\in I}\bigjoin_{k\in I}(S_i \times S_j)\bigr) \\
		             & =\bigjoin_{j\in I}\bigjoin_{k\in I}m(S_i \times S_j) 
		\leq \bigjoin_{j\in I}\bigjoin_{k\in I}S =S,
	\end{align*}
	while
	\begin{align*}
		i(S) & =i\bigl(\bigjoin_{j\in I} S_j\bigr) =\bigjoin_{j\in I} i(S_j) \leq \bigjoin_{j\in I} S_j =S,
	\end{align*}
	and
	\begin{align*}
		a(B\times S) & =a\bigl(B\times \bigjoin_{j\in I} S_j\bigr) =a\bigl(\bigjoin_{j\in I} (B\times S_j)\bigr) =\bigjoin_{j\in I} a(B\times S_j) \leq \bigjoin_{j\in I}S_j =S
	\end{align*}
	which proves that $S$ is a sub-$B$-group of $X$.
\end{proof}

The equivalence between split extensions of groups over $B$ and $B$-groups is Yoneda invariant. Hence it follows that:

\begin{proposition}\label{points=B-groups}
	If $B$ is an internal group in $\C$, then the category $\Pt{\Gp(\C)}{B}$ is
	equivalent to the category of $B$-groups in $\C$.\noproof
\end{proposition}

\begin{proposition}
	The category of internal groups in a geometric category $\C$ admits split extension cores. 
\end{proposition}
\begin{proof}
	According to Proposition~\ref{points=B-groups} it is sufficient to show that if $B$ is an internal group in $\C$, $X$ a $B$-group and $S$ a subgroup of $X$, then there is a largest sub-$B$-group of $X$ contained in $S$. Let us write $(W_j)_{j\in I}$ for the family of all sub-$B$-groups of $X$ contained in $S$. By Lemma~\ref{coherence} it follows that this is a directed family: if indeed $W_j$ and $W_k$ are sub-$B$-groups of $X$ contained in $S$, then Lemma~\ref{coherence} implies that their join as subgroups is a sub-$B$-group of $X$ contained in $S$. Lemma~\ref{directed} now implies that the join $\bigjoin_{j\in I} W_j$ in $\C$ is a sub-$B$-group of $X$ contained in $S$.
\end{proof}

\begin{theorem}\label{Theorem Internal Groups}
	The homological category of internal groups in a geometric category is always an algebraic logos.\noproof
\end{theorem}

\begin{remark}
	This immediately provides us with an alternative argument towards Example~\ref{Example Compact Hausdorff Algebras}. Unlike that example, though, Theorem~\ref{Theorem Internal Groups} does not immediately adapt to other types of algebras: while a version of Lemma~\ref{joins_of_subgroup} is still available in general, the interpretation of internal actions on which Lemma~\ref{coherence} depends is typical for groups.
\end{remark}

\begin{problem}
This technique does not seem to extend easily to a proof which is valid in the (weaker) coherent case, so (unlike what is claimed in Examples~4.19 of~\cite{acc}) we do not know whether the category of internal groups in a coherent category is algebraically coherent.
\end{problem}

\section*{Acknowledgements}

Thanks to the referee for careful comments and suggestions leading to the current, improved version of the text, and to George Peschke for fruitful discussions on the subject of the article. 

%\bibliography{tim}
%\bibliographystyle{amsplain}

\end{document}